\documentclass{article}

\usepackage{amsmath,amssymb,amsthm,hyperref}
\usepackage{geometry,graphicx}

\newtheorem{lemma}{Lemma}[section]
\newtheorem{definition}[lemma]{Definition}
\newtheorem{conjecture}{Conjecture}
\newtheorem{corollary}[lemma]{Corollary}
\newtheorem{theorem}{Theorem}
\newtheorem{observation}[theorem]{Observation}
\newcommand\ex{{\mathrm{ex}}}

\title{A characterization of edge-ordered graphs with almost linear extremal functions}
\author{Gaurav Kucheriya \thanks{Department of Applied Mathematics, Charles University, Prague, Czechia, Email: \href{mailto:gaurav@kam.mff.cuni.cz}{gaurav@kam.mff.cuni.cz}. Supported by GA\v{C}R grant 22-19073S and European Union’s Horizon 2020 research and innovation programme under the Marie Skłodowska-Curie grant agreement No. 823748.} \and G\'abor Tardos \thanks{Alfr\'ed R\'enyi Institute of Mathematics, Budapest, Hungary, Email: \href{mailto:tardos@renyi.hu}{tardos@renyi.hu}. Supported by the National Research, Development and Innovation Office, NKFIH projects K-132696 and SNN-135643 and by the ERC Advanced Grant ``GeoScape''.}}
\date{}
\begin{document}
\maketitle
\begin{abstract}
The systematic study of Tur\'an-type extremal problems for edge-ordered graphs was initiated by Gerbner et al.\ in 2020. They conjectured that the extremal functions of edge-ordered forests of order chromatic number 2 are $n^{1+o(1)}$. Here we resolve this conjecture proving the stronger upper bound of $n2^{O(\sqrt{\log n})}$. This represents a gap in the family of possible extremal functions as other forbidden edge-ordered graphs have extremal functions $\Omega(n^c)$ for some $c>1$. However, our result is probably not the last word: here we conjecture that the even stronger upper bound of $n\log^{O(1)}n$ also holds for the same set of extremal functions.
\end{abstract}

\section{Introduction}

Tur\'an-type extremal graph theory asks how many edges an $n$-vertex simple graph can have if it does not contain a subgraph isomorphic to a \emph{forbidden graph}. We introduce the relevant notation here.

\begin{definition}
We say that a simple graph $G$ \emph{avoids} another simple graph $H$, if no subgraph of $G$ is isomorphic to $H$. The Tur\'an number $\ex(n,H)$ of a \emph{forbidden} finite simple graph $H$ (having at least one edge) is the maximum number of edges in an $n$-vertex simple graph avoiding $H$.
\end{definition}

This theory has proved to be useful and applicable in combinatorics, as well as in combinatorial geometry, number theory and other parts of mathematics and theoretical computer science.

Tur\'an-type extremal graph theory was later extended in several directions, including hypergraphs, geometric graphs, vertex-ordered graphs, convex geometric graphs, etc. Here we work with \emph{edge-ordered graphs} as introduced by Gerbner, Methuku, Nagy, P\'alv\"olgyi, Tardos and Vizer in \cite{GMNPTV}. Let us recall the basic definitions.

\begin{definition}
An \emph{edge-ordered graph} is a finite simple graph $G$ together with a linear order on its edge set $E$. We often give the edge-order with an injective labeling $L: E\to\mathbb R$. We denote the edge-ordered graph obtained this way by $G^L$, in which an edge $e$ precedes another edge $f$ in the edge-order if $L(e)<L(f)$. We call $G^L$ the \emph{labeling} or \emph{edge-ordering} of $G$ and call $G$ the simple graph \emph{underlying} $G^L$.

An isomorphism between edge-ordered graphs must respect the edge-order. A subgraph of an edge-ordered graph inherits the edge-order and so it is also an edge-ordered graph. We say that the edge-ordered graph $G$ \emph{contains} another edge-ordered graph $H$, if $H$ is isomorphic to a subgraph of $G$ otherwise we say that $G$ \emph{avoids} $H$.

For a positive integer $n$ and an edge-ordered graph $H$ with at least one edge, let the Tur\'an number $\ex_<(n,H)$ be the maximal number of edges in an edge-ordered graph on $n$ vertices that avoids $H$. Fixing the \emph{forbidden edge-ordered graph} $H$, $\ex_<(n, H)$ is a function of $n$ and we call it the \emph{extremal function} of $H$.
\end{definition}

A simple but important dichotomy in classical extremal graph theory is the following:

\begin{observation}\label{obssimple}
If $F$ is a forest, then we have $\ex(n,F)=O(n)$. Otherwise, if $G$ is a simple graph that is not a forest, we have $\ex(n,G)=\Omega(n^c)$ for some $c=c(G)>1$.
\end{observation}

Let $F$ be forest on $k$ vertices. The well known Erd\H os-S\'os conjecture claims that $\ex(n,F)\le(k-2)n/2$, see \cite{Erdos}. Many special cases of this conjecture have been established, but it remains open in its full generality. However, it is easy to prove that if the minimum degree of a graph is at least $k-1$, then it contains $F$ and this implies $\ex(n,F)\le(k-2)n=O(n)$ as any graph on $n$ vertices with more than $(k-2)n$ edges contains a subgraph of minimum degree $k-1$. To see the other direction of the observation above, consider a graph $G$ that is not a forest. It contains a cycle $C_k$ and thus $\ex(n,G)\ge\ex(n,C_k)=\Omega(n^{1+1/k})$, where the bound on the extremal function of the cycle is proved by a standard application of the \emph{probabilistic method}. For even values of $k$, see stronger lower bounds on $\ex(n,C_k)$ in \cite{LUW}. For odd values of $k$, we trivially have $\ex(n,C_k)\ge\lfloor n^2/4\rfloor$ as the complete balanced bipartite graph does not contain $C_k$.

This observation gives a simple characterization of graphs with linear extremal functions and also claims the existence of a large gap in the extremal functions. Such a gap does not exist for edge-ordered graphs as there are several edge-ordered graphs (among them several distinct edge-orderings of the 4-edge path) whose extremal functions are $\Theta(n\log n)$, see \cite{GMNPTV}. Therefore one has two ways to find an analogue of Observation~\ref{obssimple} about edge-ordered graphs. First, one can try to characterize edge-ordered graphs with strictly linear extremal functions, or one can try to characterize edge-ordered graphs with \emph{almost linear} extremal function, that is, with extremal functions in the class $n^{1+o(1)}$. The former problem eluded a solution so far. We have partial results, namely characterization of \emph{connected} edge-ordered graphs with linear extremal functions in the upcoming paper \cite{KT2}. In this paper we address the latter problem, namely we prove the following conjecture of \cite{GMNPTV}:

\begin{conjecture}\label{conmain}
The extremal function of an edge-ordered graph $G$ is almost linear if and only if $G$ is an edge-ordered forest of order chromatic number 2.
\end{conjecture}

Here we use the term ``edge-ordered forest'' to denote an edge-ordered graph with a forest as its underlying simple graph. We use the terms ``edge-ordered path'' or ``edge-ordered tree'' similarly.
The notion of \emph{order chromatic number} is defined in the paper \cite{GMNPTV}. Instead of recalling the definition we recall a simple characterization of edge-ordered forests of order-chromatic number 2, also from \cite{GMNPTV} that the reader can treat as a definition for the purposes of this paper. See Figure~\ref{figclose} for an illustration.

\begin{definition}
We call a vertex $v$ of an edge-ordered graph \emph{close} if the edges incident to $v$ are consecutive in the edge-ordering, that is, they form an interval in the edge-order.
\end{definition}

\begin{lemma}[\cite{GMNPTV}]\label{ocn2}
Let $H$ be an edge-ordered forest with at least one edge. The order chromatic number of $H$ is 2 if and only if $H$ has a proper 2-coloring such that all vertices in one of the two color classes are close.
\end{lemma}

\begin{figure}\label{figclose}
		\begin{minipage}[t]{.48\linewidth}
			\centering
			\fbox{\includegraphics[scale=0.55]{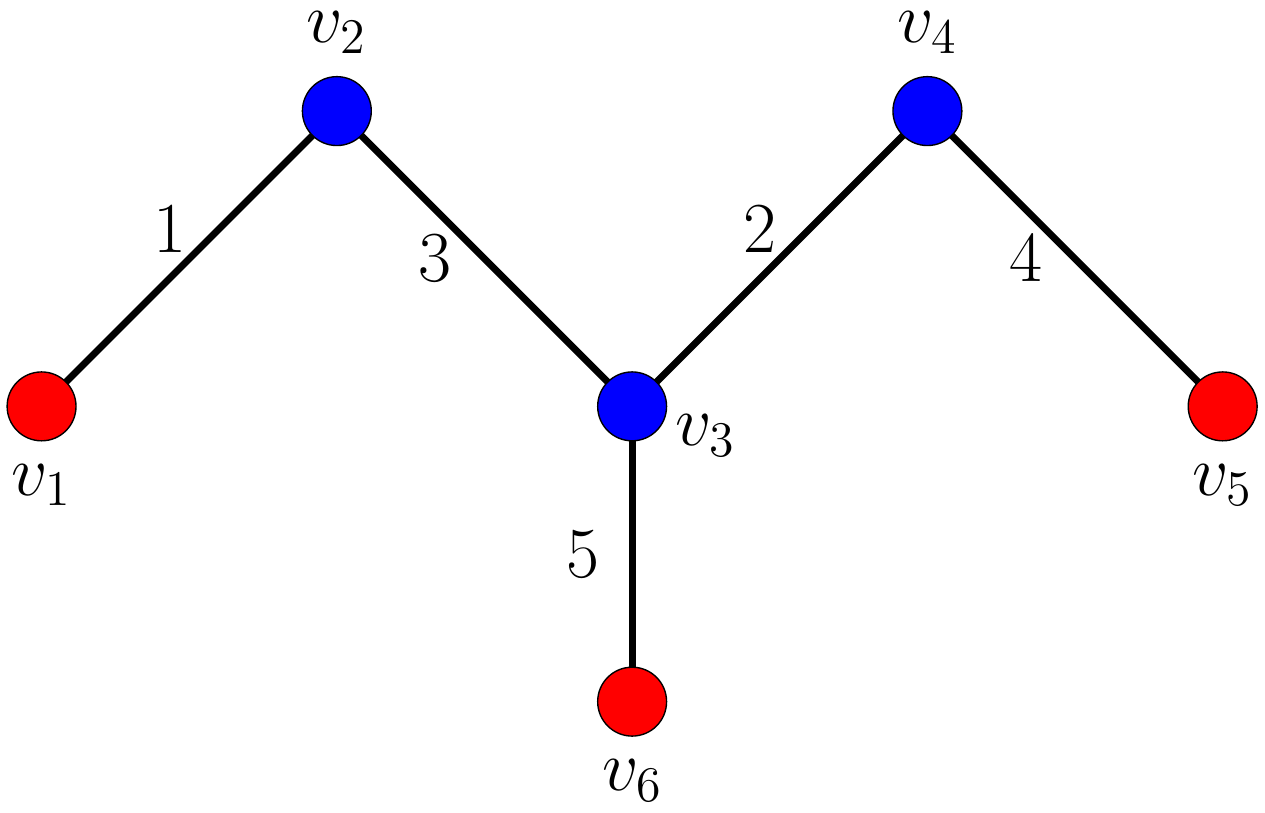}}
		\end{minipage}\qquad
		\begin{minipage}[t]{.48\linewidth}
			\centering
			\fbox{\includegraphics[scale=0.55]{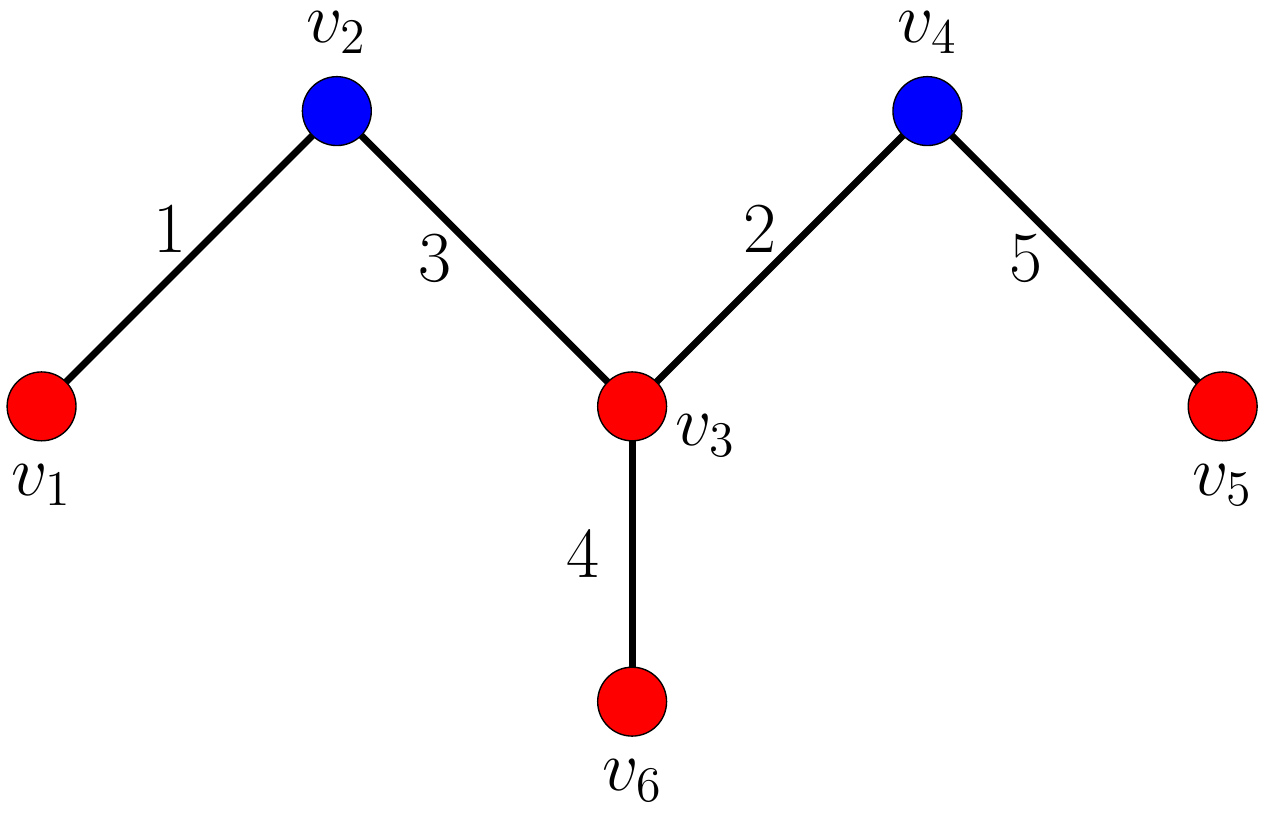}}
		\end{minipage}
	\caption{Close vertices are colored in red and the non-close in blue. The tree on the left has order chromatic number higher than $ 2 $, while the one on the right has order chromatic number $ 2 $.}
	\end{figure}

The most general result of Tur\'an-type extremal graph theory is the Erd\H os-Stone-Simonovits theorem, \cite{ES1,ES2}. It states that the extremal function of any non-bipartite forbidden simple graph is $\Theta(n^2)$ and gives the exact asymptotics in terms of the \emph{chromatic number} of the forbidden graph. An analogous result for edge-ordered graphs appeared in \cite{GMNPTV}. It states that if the forbidden edge-ordered graph has order chromatic number larger than 2, then its extremal function is $\Theta(n^2)$ and gives the exact asymptotics in terms of its order chromatic number. A notable difference between the chromatic number of simple graphs and the order chromatic number of edge-ordered graphs is that the latter is not necessarily finite: the order chromatic number of a finite edge-ordered graph is either a positive integer or infinity.

In light of the above, the ``only if'' direction of Conjecture~\ref{conmain} is self-evident. Indeed, if the order chromatic number of an edge-ordered graph is not 2, then its extremal function is quadratic (note that the order chromatic number is 1 only for edge-ordered graphs with no edges, and extremal functions are not defined for those). And if the underlying simple graph $G_0$ of the edge-ordered graph $G$ is \emph{not} a forest, then we simply have
$$\ex_<(n,G)\ge\ex(n,G_0)=\Omega(n^c)$$
for some $c=c(G)>1$. Here the first inequality trivially holds for the extremal functions of every edge-ordered graph and its underlying simple graph, while the equality is a special case of Observation~\ref{obssimple} and can be simply verified for $c=k/(k-1)$ if $G_0$ contains a $k$-cycle.

Our main result here is proving the missing ``if'' direction of Conjecture~\ref{conmain} in a stronger form:

\begin{theorem}\label{main}
If $H$ is an edge-ordered forest with order chromatic number 2, then
$$\ex_<(n,H)=n\cdot2^{O(\sqrt{\log n})}.$$
\end{theorem}

This result establishes a gap in the extremal functions of edge-ordered graphs, although this gap is smaller than the one we saw for simple graphs: if an extremal function is \emph{not} of the form $n2^{O(\sqrt{\log n})}$, then it must be of the form $\Omega(n^c)$ for some $c>1$.

The \emph{vertex-ordered} variant of Conjecture~\ref{conmain} has a longer history. Vertex-ordered graphs and their extremal functions can be defined analogously to edge-ordered graphs. Pach and the second author, \cite{PT}, proved the vertex-ordered version of the Erd\H os-Stone-Simonovits theorem in which the \emph{interval chromatic number} plays the role played by the chromatic number in the classical result. They conjectured that the extremal functions of vertex-ordered forests of interval chromatic number 2 are $O(n\log^{O(1)}n)$ or (in a weaker form of the conjecture) $n^{1+o(1)}$. Both forms of the conjecture are still open in general, but the paper \cite{PT} proves the stronger form for vertex-ordered forests of up to 6 vertices and in a more recent paper Kor\'andi, Tardos, Tomon and Weidert, \cite{KTTW} prove the weaker conjecture for a large class of vertex-ordered forests.

The history of the vertex-ordered version of the conjecture goes back even more. F\"uredi and Hajnal, \cite{FH}, introduced the extremal function of 0-1 matrices in 1992 and formulated a conjecture that turns out to be equivalent to stating that an $O(n\log n)$ bound holds for the extremal function of any vertex-ordered forest of interval chromatic number 2. This proved to be too strong as Pettie, \cite{Pettie} found a counterexample.

The rest of the paper is organized as follows. In Section~\ref2 we state and prove Theorem~\ref{step} which serves as a single step in our density increment argument proving Theorem~\ref{main}. This argument is presented in Section~\ref3. We finish the paper with some concluding remarks in Section~\ref4.

\section{The density increment step}\label2

As we have mentioned in the Introduction, we will use a density increment argument in the next section to prove Theorem~\ref{main}. We use ``density'' informally here as we do not introduce a single parameter we would call the density of an (edge-ordered) graph, but rather we concentrate on two parameters: the average degree and the number of edges. A graph with $n$ vertices and $m$ edges has average degree $d=2m/n$. Note that in this setting the smaller the number of edges are (for a given average degree) the denser we consider the graph. The following theorem represents a density increment step as passing from a graph to its subgraph we only loose a constant factor in the average degree, but decrease the size (number of edges) more significantly.

\begin{theorem}\label{step}
Let $H$ be an edge-ordered forest on $\ell$ vertices with a proper $2$-coloring such that one side consists of $k\ge2$ vertices, all of which are close. If the edge-ordered graph $G$ with $m$ edges and average degree $d>0$ avoids $H$, then $G$ has a subgraph with at most $f$ edges and average degree at least $d/(4k-4)$, where $f=\lfloor(134k\ell^2/d)^{1/(k-1)}m\rfloor$.
\end{theorem}

We prove this theorem through a series of lemmas. We start by fixing the edge-ordered forest $H$ together with $k$ and $\ell$ as in the theorem. That is, $H$ has $\ell$ vertices in total, out of which one side of a proper 2-coloring contains $k\ge2$ vertices, all of which are close. For easier reference, we call these $k$ vertices the \emph{left vertices} and the remaining $\ell-k$ vertices as the \emph{right vertices}. Since the left vertices of $H$ are close, we can enumerate them as $w_1,\dots,w_k$ such that all the edges incident to $w_i$ precede any edge incident to $w_j$, whenever $i<j$.

We also fix the edge-ordered graph $G$ with $m$ edges and average degree $d>0$ that avoids $H$. $G$ has $n=2m/d$ vertices. Let $G_0$ be its underlying simple graph and let us fix the labeling $G=G_0^L$ such that $L$ uses the integer labels from $1$ through $m$. We call a sequence of integer thresholds $t=(t_0,\dots,t_k)$ satisfying $0=t_0<t_1<\cdots<t_k=m$ a \emph{grid}. Using such a grid $t$ we classify the edges of $G_0$ into $k$ \emph{classes}: we say that an edge $e$ of $G$ belongs to class $j$ ($1\le j\le k$) if $t_{j-1}<L(e)\le t_j$. Let $G'$ be an arbitrary subgraph of $G$. For a class $j$ and a vertex $v$ of $G'$, we write $d_j(v)$ for the \emph{$j$-degree} of $v$ in $G'$, that is, for the number of class $j$ edges incident to $v$ in $G'$ (the dependence on the subgraph $G'$ and the grid $t$ is omitted from the notation). Let us define the \emph{weight of a vertex} $v$ in $G'$ as $W_{G',t}(v)=\min_{1\le j\le k}d_j(v)$. See Figure~\ref{figweight} for an illustration. We say that $v$ is \emph{heavy in $G'$} if $W_{G',t}(v)\ge \ell$. The weight of the subgraph $G'$ is defined as $W_t(G')=\sum_{v\in V(G')}W_{G',t}(v)$.

\begin{figure}\label{figweight}
		\begin{minipage}[t]{1\linewidth}
			\centering
			\fbox{\includegraphics[scale=1]{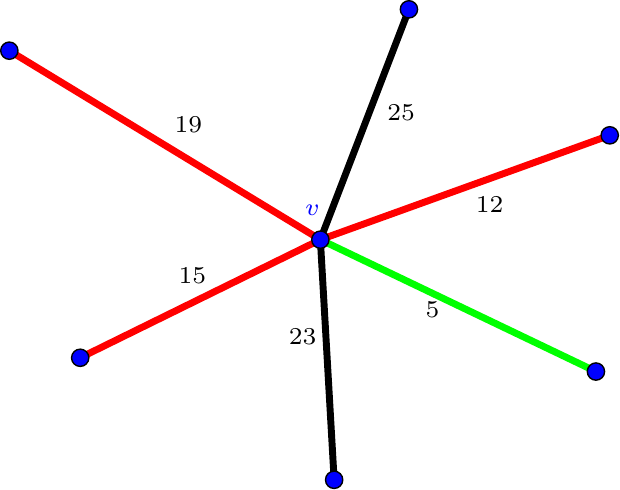}}
		\end{minipage}
	\caption{The figure depicts the neighborhood of a vertex $v$ in a subgraph $G'$ of $G$. In this example we have $k=3$, $m=30$. Choosing the grid $t$ consisting of the thresholds $0<10<20<30$ makes $d_1(v)=1$, $d_2(v)=3$, $d_3(v)=2$, and $W_{G',t}(v)=1$.}
	\end{figure}

Let $H'$ be a subgraph of $H$ and $G'$ be a subgraph of $G$ and let $t$ be a grid. We call a map $f$ a \emph{nice embedding} of $H'$ in $G'$ if $f$ maps the vertices of $H'$ to the vertices of $G'$, it induces an isomorphism between $H'$ and a subgraph of $G'$ (that is, it is injective, it maps edges to edges and it preserves the edge-order) and in addition it satisfies:
\begin{description}
\item[(i)] the image $f(y)$ of any right vertex $y$ is a heavy vertex in $G'$ and
\item[(ii)] for any edge $w_iy$ in $H'$, $f(w_i)f(y)$ is a class $i$ edge of $G'$.
\end{description}

Note that condition~(ii) implies that $f$ preserves the edge-order between edges incident to distinct left vertices. This makes it much simpler to check if a function $f$ preserves the edge-order: beyond checking condition~(ii) it is enough to compare the images of edges sharing a left vertex. This is the reason behind introducing nice embeddings, and also behind considering grids in general.

\begin{lemma}\label{largeweight}
If a subgraph $G'$ of $G$ and a grid $t$ satisfy $W_t(G')\ge2\ell(u+1)n$ for some integer $u\ge0$, then any $u$-edge subgraph $H'$ of $H$ has a nice embedding in $G'$.
\end{lemma}

\begin{proof}
We prove the lemma by induction on $u$. In the base case of $u=0$, $H'$ has no edges, so a nice embedding of $H'$ in $G'$ in this case is simply an injective mapping of the vertices of $H'$ to the vertices of $G'$ that maps right vertices to heavy vertices. The weight of any vertex $v$ of $G'$ satisfies $W_{G',t}(v)\le n$, but their sum, $W_t(G')$ is at least $2\ell n$, so at least $\ell$ vertices are heavy in $G'$. Therefore we can choose a nice embedding of $H'$ in $G'$ and moreover we can even ensure that it maps all vertices of $H'$ to heavy vertices in $G'$.

Let $u\ge 1$ and assume that the statement of the lemma holds for subgraphs $H'$ of $H$ with $u-1$ edges. Let us fix a subgraph $H'$ of $H$ with $u$ edges, a grid $t$ and a subgraph $G'$ of $G$ satisfying $W_t(G')\ge2\ell(u+1)n$. We need to show that there is a nice embedding of $H'$ in $G'$. We distinguish two cases.

First we assume that there is a left leaf vertex in $H'$, say $w_i$. Let $H''$ be the subgraph obtained from $H'$ by removing $w_i$. It has $u-1$ edges, so we can apply the inductive hypothesis to $H''$ and find a nice embedding $f$ of $H''$ in $G'$. Let $y$ be the (only) neighbor of $w_i$ in $H'$. To extend $f$ into a nice embedding of $H'$ in $G'$ all we need to do is find an image for $w_i$, a vertex $x$ outside the image of $H''$ that is connected to $f(y)$ by a class $i$ edge. Indeed, conditions~(i) and (ii) will be satisfied and the extended embedding still preserves the edge-order because the edge $xf(y)$ is the only class $i$ edge in the image of $H'$ and this determines its order among the other edges in the image $f(H')$. Clearly, $y$ is a right vertex, so by condition (i), its image $f(y)$ is heavy in $G'$, so there are at least $\ell$ other vertices connected to $f(y)$ by a class $i$ edge of $G'$. We can choose one of them outside the image of $H''$ which gives us a nice embedding of $H'$ in $G'$.

Assume now that $H'$ has no left leaf. In this case we find a right leaf $y$ of $H'$ such that its (only) neighbor $w_i$ satisfies the condition that the edge $yw_i$ is either the smallest or the largest among the edges incident to $w_i$ in $H'$. Let us first see why such a leaf $y$ exists. Start a path in $H'$ at an arbitrary non-isolated vertex and continue always on the lowest or the highest edge incident to the current vertex. Choose among these two options to avoid backtracking. As $H'$ is a forest, one eventually hits a leaf, say $y$ which must be a right vertex (since by assumption there are no left leaves) and must satisfy the condition stated above.

Let us obtain $H''$ by removing $y$ from $H'$. We will use the inductive hypothesis for $H''$ not with respect to $G'$ but with respect to a subgraph of $G'$. We call an edge $xz$ of $G'$ \emph{eligible for $z$} if it is a class $i$ edge and $x$ is heavy in $G'$. If $yw_i$ is the smallest edge at $w_i$ in $H'$, then we obtain the subgraph $G''$ by removing the $\ell$ smallest eligible edges from $G'$ for every vertex $z$ of $G'$. Similarly, if $yw_i$ is the largest edge at $w_i$ in $H'$, then we obtain the subgraph $G''$ by removing the $\ell$ largest eligible edges from $G'$ for every vertex $z$ of $G'$. In both cases, if there are fewer than $\ell$ eligible edges for a vertex we remove them all. This way we remove at most $\ell n$ edges in total. Note that as $G''$ is a subgraph of $G'$, therefore heavy vertices in $G''$ are also heavy in $G'$.

Let us consider how a single edge deletion changes the weights of the vertices. Clearly, only the weights of the two vertices connected by the deleted edge can change, and they decrease by at most $1$. Thus, removing at most $\ell n$ edges from $G'$ decreases its weight by at most $2\ell n$, so we have
$$W_t(G'')\ge W_t(G')-2\ell n\ge2\ell(u+1)n-2\ell n=2\ell un.$$

This makes the induction hypothesis applicable, so there is a nice embedding $f$ of $H''$ in $G''$. We want to extend $f$ with a well chosen image $x$ for $y$ to obtain a nice embedding of $H'$ in $G'$. We must choose a heavy vertex $x$ in $G'$ (to satisfy condition (i)) that is connected to $f(w_i)$ with a class $i$ edge (to satisfy condition (ii)). In other words, the edge $xf(w_i)$ must be eligible for $f(w_i)$. Furthermore, this vertex $x$ must be outside the image of $H''$ and (to preserve the edge-order) the edge $xf(w_i)$ must be in the correct place in the edge-order compared to the images of the edges of $H''$.

Let us consider the case when $yw_i$ is the smallest edge of $H'$ incident to $w_i$, the other case can be handled similarly. Note that $w_i$ is not a leaf as $H'$ has no left leaves. Let $y'w_i$ be the second smallest edge of $H'$ at $w_i$. Clearly, $f(y')f(w_i)$ is a class $i$ edge in $G''$ by condition~(ii) and $f(y')$ is heavy in $G''$ by condition~(i). Therefore $f(y')$ is also heavy in $G'$, so the edge $f(y')f(w_i)$ is eligible for $f(w_i)$. This is an edge in $G''$ and by the construction of $G''$ we know that $\ell$ smaller eligible edges for $f(w_i)$ were removed from $G'$. We choose $x$ such that $xf(w_i)$ is one of these removed smaller eligible edges and $x$ is outside the image of $H''$. This choice leads to a nice embedding of $H'$ in $G'$ because conditions~(i) and (ii) are satisfied and $xf(w_i)$ (being smaller than the edge $f(y')f(w_i)$) must be the smallest class $i$ edge in the image of $H'$. This finishes the proof of the inductive step and with that, the proof of the lemma.
\end{proof}

We needed the subgraphs $H'$ of $H$ and $G'$ of $G$ for the induction to work, but we will only use the following corollary based on the special case $H'=H$, $G'=G$:

\begin{corollary}\label{cor}
We have $W_t(G)<2\ell^2n$ for any grid $t$.
\end{corollary}

\begin{proof}
If $W_t(G)\ge2\ell^2n$, then $H$ has a nice embedding in $G$ by Lemma~\ref{largeweight}. This contradicts our assumption that $G$ avoids $H$.
\end{proof}  

So far we considered a fixed grid. But from now on, we consider a uniform random grid, that is, we choose a grid $t$ consisting of integer thresholds $0=t_0<t_1<t_2<\cdots<t_k=m$ uniformly from all the $\binom{m-1}{k-1}$ possible grids. We will use $E_t[\cdot]$ to denote the expectation with respect to a random grid. In order for this to make sense we assume $m\ge k$. We do not lose generality with this assumption as if $m<k$, then $f\ge m$ and therefore the statement of Theorem~\ref{step} is satisfied with choosing $G$ itself as the ``dense'' subgraph.

Let us denote the degree of a vertex $v$ in $G$ by $d_v$. We call this vertex $v$ a \emph{tame vertex} if one can cover the labels of at least $0.9d_v$ of the edges incident to $v$ by $k-1$ (well chosen) intervals, each of length $f$. Otherwise, we call $v$ a \emph{wild vertex}. Recall that $f$ is given in Theorem~\ref{step} in terms of $k$, $\ell$, $m$ and $d$.

\begin{lemma}\label{vertexweight}
Any wild vertex $v$ of $G$ satisfies
$$E_t[W_{G,t}(v)]\ge\frac{d_v(f+1)^{k-1}}{10k\binom{m-1}{k-1}}.$$
\end{lemma}

\begin{proof}

Let $ c =\left\lceil\frac{d_v}{10k}\right\rceil$. Consider the set $S$ of labels of the $d_v$ edges incident to $v$, let $a_1$ be the $c$-th smallest element in $S$ and set $I_1$ to be the interval $[a_1,a_1+f]$. Note that there are exactly $c-1$ elements of $S$ below $I_1$. For $i>1$, we consider the elements in $S$ above $I_{i-1}$, take $a_i$ to be the $c$-th smallest among them (if it exists) and set $I_i=[a_i,a_i+f]$. As before, we have exactly $c-1$ elements of $S$ between $I_{i-1}$ and $I_i$. This process ends with the interval $I_{i^*}$ if there are fewer than $c$ elements of $S$ above $I_{i^*}$. See Figure~3.

\begin{figure}[h]\label{figwild}
	\begin{minipage}[t]{1\linewidth}
		\centering
					\fbox{\includegraphics[scale=.9]{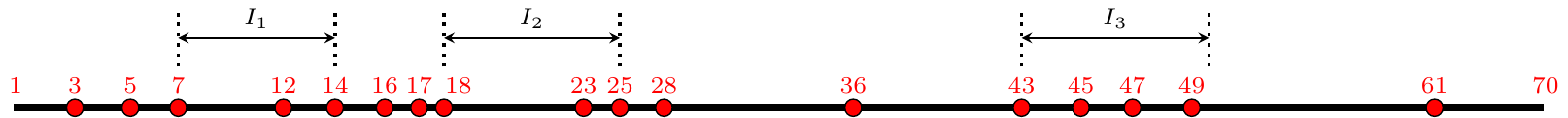}}
	\end{minipage}
	\caption{An illustration for the proof of Lemma~\ref{vertexweight}. The values marked are the labels of the edges incident to some wild vertex $v$. If we have $f=7$ and $c=3$, then the intervals $I_i$ are as shown and we have $i^*=3$ such intervals.}
\end{figure}

Consider first the case $i^*<k$. The elements of $S$ not covered by the intervals $I_1,\dots,I_{i^*}$ are either below $I_1$, above $I_{i^*}$ or between $I_i$ and $I_{i+1}$ for some $1\le i<i^*$. We saw that we have at most $c-1$ elements of $S$ in each of these $i^*+1$ categories, so we have at most $(i^*+1)(c-1)<d_v/10$ elements of $S$ not covered. As $i^*\le k-1$ this means that $v$ is tame, contrary to our assumption.

So we must have $i^*\ge k$. Consider the grids $t=(t_0,\ldots,t_k)$ where $t_0=0$, $t_k=m$ and $t_i$ is an integer from $I_i$ for $1\le i<k$. We can choose these intermediate values $t_i$ in exactly $f+1$ ways, so we have $(f+1)^{k-1}$ such grids. For each such grid $t$, the edges with labels up to and including $a_1$ are class 1 edges. As $a_1$ is the $c$-th smallest value in $S$, we have at least $c$ such edges incident to $v$ making the 1-degree of $v$ at least $c$. Similarly, for $1<i\le k$, all the edges with labels above $I_{i-1}$ but at most $a_i$ are necessarily class $i$ edges. As $a_i$ is the $c$-th smallest element of $S$ above $I_{i-1}$ we have at least $c$ class $i$ edges adjacent to $v$. This makes $W_{G,t}(v)\ge c$ for each of the grids considered. So by Markov's inequality we get $E_t[W_{G,t}(v)]\ge\frac{(f+1)^{k-1}}{\binom{m-1}{k-1}}c\ge\frac{d_v(f+1)^{k-1}}{10k\binom{m-1}{k-1}}$ as needed.
\end{proof}

Let us now fix a suitable set of intervals for every tame vertex of $G$. So we have $k-1$ intervals of length $f$ associated to each tame vertex, and they collectively cover the labels of at least $0.9$ fraction of the incident edges. We define a subgraph $G^*$ of $G$ as follows. The vertices of $G^*$ are the tame vertices of $G$, the edges of $G^*$ are the edges $uv$ of $G$ such that both $u$ and $v$ are tame and the label $L(uv)$ is covered by an interval associated to $u$ and also by a (possibly different) interval associated to $v$. Note that our construction ensures that the labels of each star subgraph in $G^*$ are covered by the intervals associated to the center vertex.

\begin{lemma}\label{gstar}
$G^*$ has at least $m/2$ edges.
\end{lemma}

\begin{proof}
Instead of counting edges in $G^*$, we count edges of $G$ outside $G^*$. For each such edge $e$ we \emph{blame} one of the vertices connected by $e$: either a wild vertex or a tame vertex with none of the intervals associated to it covering $L(e)$.

If $v$ is tame, the associated intervals cover the labels of $0.9d_v$ edges, so we blame $v$ for the exclusion of at most $0.1d_v$ edges. All together, tame vertices are blamed for the exclusion of at most $\sum_{v\mathrm{\ tame}}0.1d_v\le0.1\sum_vd_v=0.2m$ edges.

By Corollary~\ref{cor} we have $W_t(G)<2\ell^2n$ for any grid $t$. So this holds for the expected weight too, giving the first inequality below:
\begin{eqnarray*}
2\ell^2n&>&E_t[W_t(G)]\\
&=&\sum_vE_t[W_{G,t}(v)]\\
&\ge&\sum_{v\mathrm{\ wild}}E_t[W_{G,t}(v)]\\
&\ge&\sum_{v\mathrm{\ wild}}\frac{d_vf^{k-1}}{10k\binom{m-1}{k-1}}\\
&=&\frac{f^{k-1}}{10k\binom{m-1}{k-1}}\sum_{v\mathrm{\ wild}}d_v,
\end{eqnarray*}
where the equation in the second line is the linearity of expectation, and the inequality in the fourth line comes from Lemma~\ref{vertexweight}.

If $v$ is wild, it is blamed for the exclusion of all $d_v$ incident edges. Thus, wild vertices are blamed for the exclusion of at most
$$\sum_{t\mathrm{\ wild}}d_v\le\frac{20k\ell^2n\binom{m-1}{k-1}}{f^{k-1}}<0.3m$$
edges, where the first inequality comes from the calculation in the previous paragraph, while the second inequality is guaranteed by the coarse bound $\binom{m-1}{k-1}<m^{k-1}$ and the choice of $f$ in the statement of Theorem~\ref{step}.

As the exclusion of at most $0.2m$ edges from $G^*$ is blamed on tame vertices and the exclusion of at most $0.3m$ edges is blamed on wild vertices we must have at least $m/2$ edges not excluded.
\end{proof}

All we need to finish the proof of Theorem~\ref{step} is Lemma~\ref{gstar}.

\begin{proof}[Proof of Theorem~\ref{step}.]
Let us summarize what we proved so far. We fixed $H$ and $G$ as in the theorem and further fixed a labeling $L$ giving the labels $1,\dots,m$ to the $m$ edges of the underlying simple graph $G_0$ of $G$ making $G=G_0^L$. We identified a subgraph $G^*$ of $G$ where the labels of the edges in each star subgraph can be covered by $k-1$ intervals of length $f$ each. Finally, we proved Lemma~\ref{gstar} stating that $G^*$ has at least $m/2$ edges.

We partition $G^*$ into subgraphs as follows: For $1\le j\le \lceil m/f\rceil$, we define $G^*_j$ to be the subgraph of $G^*$ consisting of the edges $e$ with $(j-1)f<L(e)\le jf$. A vertex $v$ of $G^*$ is a vertex of $G^*_j$ if at least one of the edges incident to $v$ is in $G^*_j$. Note that the edges with labels from an interval of length $f$ can show up in at most two different subgraphs $G^*_j$, therefore any vertex of $G^*$ shows up in at most $2k-2$ subgraphs $G^*_j$.

Clearly, $G^*_j$ is a subgraph of $G$ with at most $f$ edges. To finish the proof of the theorem we only need to establish that one of them has high enough average degree. For this, consider the disjoint union $Z$ of all the graphs $G^*_j$. Each edge of $G^*$ appears exactly once in $Z$, so we have $|E(Z)|=|E(G^*)|\ge m/2$. Each of the vertices of $G^*$ has at most $2k-2$ copies in $Z$, so we have $|V(Z)|\le(2k-2)|V(G^*)|\le(2k-2)n$. This makes the average degree of $Z$ at least $m/((2k-2)n)=d/(4k-4)$. At least one of the constituent graphs $G^*_j$ has average degree at least as high as their disjoint union $Z$, finishing our proof.
\end{proof}

\section{Proof of Theorem~\ref{main}}\label3

To prove Theorem~\ref{main} we will use Theorem~\ref{step} recursively. This is a standard calculation, but we write it down in full details to be self-contained. We make no effort to optimize the constants.

Let $H$ be an edge-ordered forest of order chromatic number 2. By Lemma~\ref{ocn2} it satisfies the conditions of Theorem~\ref{step} for the appropriate values of the parameters $k$ and $\ell$ unless the number of vertices on one side of the bipartition is $k=1$. However, if $k=1$, then $H$ is an edge-ordered star (with a few isolated vertices possibly added) and therefore it has a single edge-ordering up to isomorphism. In this case its extremal function agrees with the extremal function of its underlying simple graph, which is linear. Therefore we can assume $k>1$ and then Theorem~\ref{step} does apply to $H$.

Let $G_0$ be an arbitrary edge-ordered graph avoiding $H$. If $G_0$ has $n_0$ vertices and $m_0\ge1$ edges, then it has average degree $d_0=2m_0/n_0$. To simplify our calculation we will introduce various constants $c_i$ depending on $H$ and not depending on $G_0$. By Theorem~\ref{step}, $G_0$ contains a subgraph $G_1$ with $m_1\le c_1m_0/d_0^{c_2}$ edges and average degree $d_1\ge d_0/c_3$, where $c_1=(134k\ell^2)^{1/(k-1)}>1$, $c_2=1/(k-1)>0$ and $c_3=4k-4>1$. Any subgraph of $G_0$ must also avoid $H$, so we can apply Theorem~\ref{step} recursively: for $t\ge0$ let $G_{t+1}$ be a subgraph of $G_t$ as specified in the theorem. If we write $d_t$ for the average degree of $G_t$ and $m_t$ for the number of edges in $G_t$, we have
$$d_{t+1}\ge d_t/c_3\hskip2cm m_{t+1}\le c_1m_t/d_t^{c_2}$$
for every $t\ge0$. This recursion solves to
$$d_t\ge\frac{d_0}{c_3^t}\hskip2cm m_t\le\frac{c_1^tc_3^{c_2\binom t2}}{d_0^{c_2t}}m_0.$$
The average degree cannot exceed the number of edges, so $m_t\ge d_t$ holds for all $t$. This inequality can be rearranged to obtain
$$d_0^{c_2t}\le\frac{m_0}{d_0}\cdot(c_1c_3)^t\cdot c_3^{c_2\binom t2}.$$
This implies that at least one of the following three inequalities must hold for every $t$:
\begin{eqnarray}
d_0^{c_2t/3}&\le&\frac{m_0}{d_0},\label a\\
d_0^{c_2t/3}&\le&(c_1c_3)^t, \label b\\
d_0^{c_2t/3}&\le&c_3^{c_2\binom t2}.\label c
\end{eqnarray}
We choose $t=\left\lceil\frac{2\log d_0}{3\log c_3}\right\rceil$. Here and further in this calculation we use $\log$ to denote the binary logarithm. We need $t\ge1$ for our calculations, so here we assume $d_0>1$. Our choice ensures that (\ref c) does not hold. If (\ref b) holds, we have $d_0\le(c_1c_3)^{3/c_2}$ and therefore
\begin{equation}\label{d}
m_0=n_0d_0/2\le c_4n_0,
\end{equation}
with $c_4=(c_1c_3)^{3/c_2}/2$. Note that as $c_4>1/2$, this bound holds also in the case $d_0\le1$ that we earlier excluded. If (\ref a) holds, inserting the value of $t$ we obtain
$$\frac{n_0}2=\frac{m_0}{d_0}\ge 2^{\frac{2c_2\log^2d_0}{9\log c_3}}.$$
Rearrangement gives
$$\log d_0\le c_5\sqrt{\log n_0},$$
for $c_5=\sqrt{\frac{9\log c_3}{2c_2}}$. This yields
\begin{equation}\label{e}
m_0=n_0d_0/2\le n_02^{c_5\sqrt{\log n_0}}.
\end{equation}

As either (\ref d) or (\ref e) must hold for any edge-ordered graph $G_0$ having $n_0$ vertices and $m_0$ edges, we have
$$\ex_<(n,H)\le\max(c_4n,n2^{c_5\sqrt{\log n}})=n2^{O(\sqrt{\log n})}$$
as needed.

\section{Concluding remarks}\label4

The main open problem related to this paper is whether our Theorem~\ref{main} can be improved. We believe it can be improved and so we formulate the following conjecture that is the edge-ordered analogue of a similar conjecture on vertex-ordered graphs from \cite{PT}.

\begin{conjecture}\label{stronger}
$\ex_<(n,H)=n\log^{O(1)}n$ holds for all edge-ordered forests $H$ of order chromatic number $2$.
\end{conjecture}

Such strong upper bounds were proved for several small edge-ordered paths in \cite{GMNPTV}, including all 4-edge edge-ordered paths. In particular, an upper bound $O(n\log n)$ was proved for the extremal function of all edge-ordered 4-edge paths of order chromatic number $2$ but for the equivalent exceptional cases of the path $P_5^{1342}$ and $P_5^{4213}$. (The upper index should be interpreted as the list of labels along the path.) For the extremal function of these exceptional edge-ordered paths an $O(n\log^2n)$ bound was proved. In an upcoming paper \cite{KT2} we extend this study to most edge-ordered 5-edge paths of order chromatic number 2, proving an upper bound of either $O(n)$, $O(n\log n)$ or $O(n\log^2 n)$ for their extremal functions, but (up to equivalence) four exceptional edge-ordered 5-edge paths: $P_6^{14523}$, $P_6^{15423}$, $P_6^{31254}$ and $P_6^{14532}$ eluded our efforts and for their extremal functions only the weaker upper bounds from this paper apply.

From the other direction, we know many edge-ordered forests of order chromatic number $2$ (including several edge-ordered 4-edge paths) whose extremal function is $\Omega(n\log n)$ (see \cite{GMNPTV}) but we do not have a single example with a extremal function higher than that. Learning from Pettie's refutation, \cite{Pettie} of a conjecture of F\"uredi and Hajnal, \cite{FH} (see the context in the Introduction), we refrain from making too strong a conjecture here, instead we ask whether an edge-ordered version of Pettie's example can be found, that is an edge-ordered tree of order chromatic number 2 with extremal function $\Omega(n\log n\log\log n)$ (or higher)?

\section*{Acknowledgements}
We are grateful for the comments of Mykhaylo Tyomkyn improving the presentation of this paper. The ERC advanced grant ``GeoSpace'' made it possible for the first author to visit R\'enyi Institute in the fall of 2021 and helped enormously with the collaboration resulting in this paper.

\end{document}